\documentclass[12pt,a4paper]{article}

\usepackage{amssymb, amsmath, amsthm}

\usepackage[cm]{fullpage}
\usepackage[english]{babel}
\usepackage[cp1250]{inputenc}
\usepackage[T1]{fontenc}
\usepackage[pdftex]{graphicx}
\usepackage{booktabs}
\usepackage{bbm}
\usepackage{{enumerate}}

\usepackage[svgnames]{xcolor}
\newcommand*{\greysquare}{\textcolor{gray}{\blacksquare}}
\newcommand*{\greytriangle}{\textcolor{gray}{\blacktriangledown}}
\newcommand*{\greycircle}{\textcolor{gray}{\bullet}}
\newtheorem{thm}{Theorem}
\newtheorem{lem}{Lemma}

\newtheorem{cor}{Corollary}
\newtheorem{rem}{Remark}

\title{Asymptotic behaviour of a solution to a nonlinear equation modelling capillary rise}
\author{\L ukasz P\l ociniczak\thanks{Faculty of Pure and Applied Mathematics, Wroc{\l}aw University of Science and Technology, Wyb. Wyspia{\'n}skiego 27, 50-370 Wroc{\l}aw, Poland},
	Mateusz \'Swita{\l}a\footnotemark[1]$\;^,$\footnote{\underline{Corresponding Author}, e-mail: mateusz.switala@pwr.edu.pl}}
\date{}

\begin{document}
\maketitle

\begin{abstract}
	We are concerned with the asymptotics and perturbation analysis of a singular second-order nonlinear ODE that models capillary rise of a fluid inside a narrow vertical tube. We prove the convergence of the exact solution to a unperturbed solution when a nondimensional parameter decrease to zero. Furthermore, we provide an accurate method of approximating the asymptotic solution for large times. Due to the a fact that the nonlinear component in the main equation does not satisfy the Lipschitz continuity condition the methods used to prove the main theorems are nonstandard, require careful analysis, and can be useful in dealing with similar nonlinear ODEs. \vspace{2em}\\
	
\end{abstract}

\section{Introduction}

Capillary rise is an extraordinary physical phenomenon that is ubiquitous in nature and the effects of which we can notice in many everyday situations. This natural phenomenon is mainly responsible for water distribution in plants, in particular it causes the main driving force in the process of water supply to the upper branches of trees. Capillary rise can also be observed in porous media such as soils and rocks. In addition to its obvious natural applications, the discussed phenomenon appears also in industry (see \cite{zhmud} and \cite{cheng11}), e.g. in the printing processes and in a technique called thin layer chromatography to separate mixtures from substances. The appropriate mathematical models describing the behavior of capillary rise are gaining in importance in view of the applications of a capillary flow in industry.

 One of the first equations modelling the dynamics of capillary rise was introduced by E. W. Washburn in 1921 (see \cite{wash1921}). However, after subsequent researches and conducted experiments by others scientists, many improved mathematical models were proposed in order to provide the more accurate description of that extraordinary physical phenomenon.  In our previous work \cite{Plo18} we provided brief explanation of each component in the main governing equation modelling the capillary rise process in a narrow vertical tube 
 \begin{equation}
	\frac{8 \nu }{r^2} h h'+\rho g h+\rho \frac{\,d}{\,d t}\left(h h'\right)=\frac{2 \gamma \cos{\theta}}{r},
	\label{eqn:gov1}
	\end{equation} 
	where $h=h(t)$ is the liquid column height at time $t$, $r$ is a radius of the capillary tube, and $\mu,\, \rho$ denote viscosity and density of liquid respectively. As usual, $g$ denotes the gravitational acceleration and $\gamma$ denotes the surface tension. It is worth emphasis that in the current work the contact angle, here denoted by $\theta$, preserves a constant value during the flow. Nevertheless, some results \cite{pop,blake,joos} indicates the dependence of contact angle on the velocity of flow. The meticulous investigation of the \eqref{eqn:gov1} with various models of the velocity dependent contact angle will be an object of study in further works. Moreover, we assume that the narrow tube is initially filled with air and the walls of tube are smooth. For simplification and clarity of the main process, we do not take into account the influence of the container walls where the narrow tube is immersed. The comprehensive derivation of above governing equation can be found in \cite{masoodi12,das13, ramon08} and \cite{Plo18}. In our analysis we assume that the initial height is equal to zero and to complete the description of the governing equation we have to choose a proper initial condition for velocity of the flow. In order to ensure the consistency of the problem we have to take $h'(0)=0$. The discussion on the different choice of initial condition as well as the extensive analysis for more physical initial condition can be find in \cite{zhmud,loren2002,szekely71}. Furthermore, the rigorous mathematical analysis of the governing equation \eqref{eqn:gov1} for both cases of initial conditions Reader can find in \cite{Plo18}. Recently, a very interesting geometrical analysis of (\ref{eqn:gov1}) appeared in \cite{Zha18} and the Reader is invited to compare these two different approaches.
	
	 As was shown in various experiments, the flow in a narrow tube may not be monotone \cite{loren2002, quere1997}. Thus a two regimes are distinguished where the fluid column height either monotonically increase or oscillates near the Jurin's height (the point, where the capillary pressure balances the weight of a fluid). The change of the behavior of flow from monotone to oscillatory for certain values of nondimensional parameter was also proved analytically (see \cite{Plo18}). Nonlinear oscillations are very often difficult to analyze due to its complexity and unique character, hence, a non standard methods have to be used to analyze such a phenomenon efficiently. Nonlinear oscillations modelled by differential equations has been extensively analyzed in many papers (see for ex. \cite{Kam78,Phi89,Kon99}). A nonlinear dynamical systems are also widely studied in two fundamental monographs \cite{Guc13} and \cite{Nay08}. That monographs concern mostly the theoretical and applicable mathematical results related to the dynamical systems and nonlinear oscillations. 
	 
	 To proceed further with the asymptotic analysis we have to transform \eqref{eqn:gov1} into a more convenient form. After reducing the governing equation into a dimensionless form (see \cite{masoodi12, Plo18}) and using a simple transformation for both independent variable and height function (for more details see \cite{Plo18}) we are able to write \eqref{eqn:gov1} in following form
	 \begin{equation}
		u''(s)+\frac{1}{\sqrt{\omega} }u'(s)+\sqrt{2 u(s)}=1,\\
		\label{eqn:eqa}
		\end{equation} 
		with 
		\begin{equation}
		u(0)=0, \quad u'(0)=0,
		\label{eqn:IC3}
		\end{equation}
		where, prime denotes a differentiation with respect to the independent variable. In subsequent analysis we will rather use  parameter $\epsilon=1/\sqrt{\omega}$ instead of $\omega$. Whereas, the parameter $\omega$ is defined as follows
		\begin{equation}
	\omega = \frac{\rho ^3 g^2 r^5}{128 \mu ^2 \gamma \cos\theta}.
	\label{eqn:omega}
	\end{equation}
	    Obviously, we are not able to solve \eqref{eqn:eqa} explicitly. Hence, a advanced approximation methods need to be used to analyzed the above equation and its solution in most accurate manner.

	The following article is composed of the two main sections concerning the perturbation with respect to $\epsilon\rightarrow 0^+$ and asymptotic behavior for large times for \eqref{eqn:eqa}. In the next section we use a perturbation analysis to investigate the solution of the main equation for small values of $\epsilon$. Due to appearance of the non-Lipschitz function in the main equation we can not use any standard theorem to prove the convergence. Therefore, a new approach was proposed to prove that the exact solution approaches to the zeroth order approximation as $\epsilon\to 0^+$.  Moreover, an asymptotic behavior of the solution for small times is also presented.
	
	In Section \ref{SectionAsym} we present an asymptotic analysis of the oscillatory behavior for \eqref{eqn:eqa} for large times. For this reason we consider only the values of dimensionless parameter $\epsilon$ for which the solution oscillates around the Jurin's height ($\epsilon < 2$). We introduce a convergent method of finding a successive approximations of the solution to \eqref{eqn:eqa}. The numerical analysis revealed that for a certain values of $\epsilon$ the error between the exact solution and the approximate one might be negligible even for the relatively small values of independent variable.
	
\section{Perturbation analysis for small values of $\epsilon$}
Let us now recall the main equation with appropriate initial conditions
\begin{equation}\label{eq:GovFinal}
    u''(s)+\epsilon u'(s)+\sqrt{2 u(s)}=1,\quad u(0)=0,\, u'(0)=0.
\end{equation}
Using regular perturbation methods with respect to the small parameter $\epsilon$ we can obtain the zeroth order approximation to the exact solution of \eqref{eq:GovFinal}
\begin{equation}\label{eq:ZeroOrder}
    u_0(s)=\frac{1}{72}\sum\limits_{i=1}^{\infty}(s-6(i-1))^2 (6 i-s)^2\mathbbm{1} _{\{s\in [6(i-1),6 i]\}}.
\end{equation}
The above function is depicted on the Figure \ref{fig:ZerothOrder}. We can see that it represents an oscillation with a period equal to $6$.

\begin{figure}
	\centering
	\includegraphics[scale=1.3]{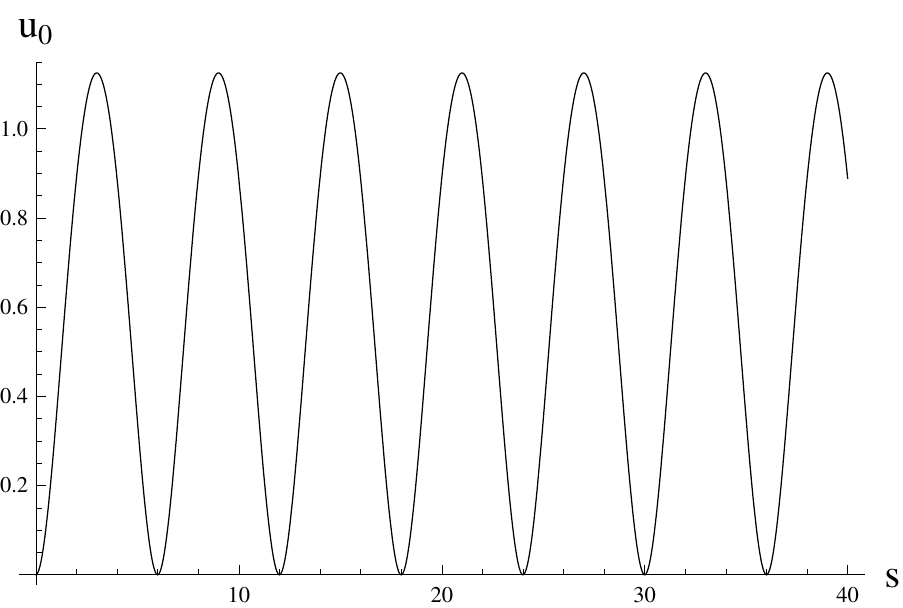}
	\caption{The leading order perturbation expansion of the solution to (\ref{eq:GovFinal}) for $\epsilon\rightarrow 0$.}
	\label{fig:ZerothOrder}
\end{figure}

To see how to obtain $u_0$ first we put $\epsilon=0$ into \eqref{eq:GovFinal} and get a second order equation without the derivative term
\begin{equation}\label{eq:UndampedU}
    u_0''+\sqrt{2 u_0}=1, \quad u_0(0)=0,\quad u_0'(0)=0.
\end{equation}
Since the above does not explicitly contain the independent variable we can introduce another function $g(u_0):=u_0'$ and using $u_0''=g \, d g/d u_0$ we are able to reduce the order what leads to
\begin{equation}
    g \frac{d g}{d u_0}+\sqrt{2 u_0}=1, \quad g(0)=0,
\end{equation} 
where $g$ is now the sought function while $u$ is treated as an independent variable. We can easily integrate the above equation to get $g(u_0)$. Then, using the definition of $g$ we have
\begin{equation}
    u_0'(x)=\pm \sqrt{\frac{2}{3}} \sqrt{-\biggl(2 \sqrt{2} \sqrt{u(x)}-3\biggr) u_0(x)}, \quad u_0(0)=0.
\end{equation}
If we integrate above equation we get two solutions but only one of them is oscillatory for $\epsilon=0$ with $u_0(s)\ge 0$ and $s\ge 0$. Thus, we have
\begin{equation}
    u_0(s)=\frac{s^2 (6-s)^2}{72}\mathbbm{1} _{\{s\in [0,6 ]\}}.
\end{equation}

To verify the periodicity of \eqref{eq:ZeroOrder} we multiply \eqref{eq:GovFinal} by $u'$ and integrate to get the energy equation
\begin{equation}\label{eq:Energy}
    \frac{u'^2}{2}+\frac{2 \sqrt{2}}{3}u^{3/2}-u=-\epsilon \int_{0}^{s}u'(x)^2 d\, x
\end{equation}
On the left hand side we have a sum of the kinetic and potential energy. It is easy to notice that the potential energy has a minimum. Moreover, for the boundary points of subintervals given in  \eqref{eq:ZeroOrder} for all $i$, the minimum value of potential energy is not achieved. Hence, the solution will oscillate with decaying amplitude (for a different treatment see \cite{Zha18}).

The above considerations were only formal and we do not know whether the exact solution of \eqref{eq:GovFinal} converges to $u_0(s)$ as $\epsilon \to 0^+$. The standard convergence theorems from perturbation theory cannot be applied here because the Lipschitz condition for nonlinearity is not satisfied. As we will see with a suitable regularization we are able to avoid this requirement. 

Before we proceed to the main result we will derive certain inequalities that will be useful in subsequent analysis.  The following Lemma provides an asymptotic behavior of function $u(s)$ and its derivative $u'(s)$ for $s\to 0^+$.
\begin{lem}\label{Lemma1}
	If $u(s)$ is a solution of \eqref{eq:GovFinal} then for  $s\in [0,2/(1+\epsilon)]$ we have the following estimates.
	\begin{enumerate}[a)]
		\item\label{ad:1} Function $u$:
	\begin{equation}\label{ineq:SmallValues}
	    \frac{1}{3}\frac{s^2}{2}\le u(s)\le \frac{s^2}{2}.
	\end{equation}
	\item\label{ad:2} The derivative $u'$:
	\begin{equation}\label{ineq:SmallValuesDiff}
	    u'^2(s)\le s^2-\frac{2 \sqrt{2}}{9}s^3.
	\end{equation}
	\end{enumerate}
	Furthermore, for $s\to 0^+$ the following asymptotic behavior holds
\begin{enumerate}
    \item[i)] Function $u$:
    \begin{equation}\label{AsymU}
        u(s)=\mathcal{O}(s^2),\quad \textrm{as}\quad s\to 0^+.
    \end{equation}
    \item[ii)] The derivative $u'$:
    \begin{equation}
        u'(s)=\mathcal{O}(s),\quad \textrm{as}\quad s \to 0^+.
    \end{equation}
\end{enumerate}
\end{lem}
\begin{proof}
	From \cite{Plo18} we have 
	\begin{equation}
	    \left(1-\frac{1}{3}(1+\epsilon)s\right)\frac{s^2}{2}\le u(s)\le \frac{s^2}{2} \quad \textrm{for}\quad s\in\left[0,\frac{3}{1+\epsilon}\right].
	\end{equation}
	We see that the function of $s$ on the leftmost side attains its maximal value at $s=2/(1+\epsilon)$. Moreover, $1-\frac{1}{3}(1+\epsilon)s$ is a decreasing function of $s$ hence we have 
	\begin{equation}
	    \frac{1}{3}=\left(1-\frac{1}{3}(1+\epsilon)\left(\frac{2}{1+\epsilon}\right)\right)\leq\left(1-\frac{1}{3}(1+\epsilon)s\right)\quad \textrm{for}\quad s\in \left[0,\frac{2}{(1+\epsilon)}\right].
	\end{equation}
	Thus \eqref{ineq:SmallValues} follows from the above. From \eqref{eq:Energy} we write
	\begin{equation}
	    \frac{u'^2}{2}+\frac{2 \sqrt{2}}{3}u^{3/2}-u\le 0.
	\end{equation}
	and using \eqref{ineq:SmallValues} we arrive at
	\begin{equation}
	    u'^2(s)\leq 2 u-\frac{4 \sqrt{2}}{3}u^{3/2}\le s^2-\frac{2 \sqrt{2}}{9}s^3.
	\end{equation}
	Specifically, we have
	\begin{equation}
	    |u'(s)|\leq s,
	\end{equation}
	what completes the first part of proof.
	
	Now, it immediate to see that \eqref{AsymU} is a consequence of \eqref{ineq:SmallValues}. Next, let us consider the derivative $u'(s)$ for small argument $s$. From \eqref{ineq:SmallValuesDiff} we have
    \begin{equation}\label{ineq:AsymDu}
        |u'(s)|\le \sqrt{s^2-\frac{2\sqrt{2}}{9}s^3}\le s,
    \end{equation}
    for $s\in [0,2/(1+\epsilon)]$. We see that the absolute value of derivative of $u$ is bounded by $s$ with $u'(0)=0$. Now, looking on \eqref{eq:GovFinal} we notice that for a very small argument the second derivative is positive which implies that the function $u'$ must be a  increasing function on the interval $[0,\delta]$ for certain $\delta>0$. Let us rewrite \eqref{eq:GovFinal} in the energy form for $s=\delta$
    \begin{equation}\label{eq:EnergyAsym}
    \frac{1}{2}u'(\delta)^2+\frac{2\sqrt{2}}{3}u(\delta)^\frac{3}{2}-u(\delta)=-\epsilon \int\limits_{0}^{\delta}u'(t)^2\, d t.
    \end{equation}
    Taking advantage of the fact that function $u'(t)$ is increasing on the interval $[0,\delta]$ we can write
    \begin{equation}
      -\epsilon \int\limits_{0}^{\delta}u'(t)^2\, d t \ge -\epsilon \delta u'(\delta)^2. 
    \end{equation}
    Hence, combining the above with (\ref{eq:EnergyAsym}) along with using (\ref{ineq:SmallValues}) we obtain
    \begin{equation}
        u'(\delta)^2\ge \frac{1}{3(\frac{1}{2}+\epsilon \delta)}(\delta^2-\delta^3),
    \end{equation}
    and finally 
    \begin{equation}
        u'(\delta)\ge M_1 \frac{1}{\sqrt{3(\frac{1}{2}+\epsilon \delta)}}\delta,
    \end{equation}
    where $M_1$ is a positive constant, $\delta \ll 1$ and $\delta\le 1- M_1^2$. From \eqref{ineq:AsymDu} and the above inequality it follows that
    \begin{equation}
        M_1 \frac{1}{\sqrt{3(\frac{1}{2}+\epsilon \delta)}}\delta\le u'(\delta)\le \delta,
    \end{equation}
    for $\delta \to 0^+ $ what completed the proof.
\end{proof}
Now, we are able to introduce the main result in which we show that the solution of (\ref{eq:GovFinal}) converges uniformly to (\ref{eq:ZeroOrder}) as $\epsilon \to 0^+$. 
\begin{thm}\label{thm:Error}
	Let $u(t)$ be the solution of \eqref{eq:GovFinal}. Then, the following asymptotic behaviour holds
	\begin{equation}\label{TheoremResult}
	\lim\limits_{\epsilon\to 0^+}|u(s)-u_0(s)|=0,
	\end{equation}
	uniformly for $s\in [0,T]$ for any fixed $T>0$.  
\end{thm}
\begin{proof}
	Let us consider the equation \eqref{eq:GovFinal} on the interval $[0, T]$, where $T$ is any fixed value greater than zero. Write 
	\begin{equation}\label{eq:Exact}
	u(s)=u_0(s)+R(s),
	\end{equation}
	where function $u_0$ is defined as in \eqref{eq:ZeroOrder} and function $R(s)$ is a unknown remainder of a zeroth order approximation to the exact solution of \eqref{eq:GovFinal}. We will find an upper bound on $R(s)$. We substitute \eqref{eq:Exact} into \eqref{eq:GovFinal} and obtain following equation
	\begin{equation}\label{eq:RemainderDiff}
	R''(s)=\sqrt{2 u_0(s)}-\sqrt{2(u_0(s)+R(s))}.
	\end{equation}
	We can easily transform above to the form of integral equation
	\begin{equation}\label{eq:Remain}
	R(s)=\int_{0}^{s}G(s-x)f(u_0(x),R(x))\,d x
	\end{equation} 
	where 
	\begin{flalign*}
	G(z)& = \frac{1}{\epsilon}\biggl(1-e^{-\epsilon z}\biggr),\\
	f(u_0(s),R(s))& =\sqrt{2 u_0(s)}-\sqrt{2(u_0(s)+R(s))}-\epsilon u_0'(s).
	\end{flalign*} 
	Thereafter we divide the interval $[0,T]$ into $[0,\delta]$ and $[\delta, T]$. We assume that the number $\delta$ is small. Its exact value will be found later.
	
	Firstly, let us consider the interval $[0,\delta]$. We can estimate the
	integral in \eqref{eq:Remain} from the above as follows
	\begin{equation}
	|R(s)|\le \int_{0}^{\delta}|G(s-x)||f(u_0(x),R(x))|\,d x\le \delta \int_{0}^{\delta}|f(u_0(x),R(x))|\,d x.
	\end{equation}
	It is a matter of simple calculus exercise to see that $u_0$ is bounded by $1$. Moreover, from \cite{Plo18} we know that the exact solution of \eqref{eq:GovFinal} is no greater than $9/8$, hence the function $f$ can be easily bounded from above by $3/2+\epsilon$. Whereas, the function $G(z)$ is concave which together with $G(s-x)\le G(\delta)$ gives $G(s-x)\le \delta$. Therefore, 
 \begin{equation}\label{ineq:IntervalSmall}
 |R(s)|\le \biggl(\frac{3}{2}+\epsilon\biggr)\delta^2, \quad s\in[0,\delta].
 \end{equation}
 Next, let us consider subinterval $[\delta, T]$. To this end, define a new function $w_\delta$ which satisfies the following initial value problem
 \begin{equation}\label{eq:FunctionW}
 w_\delta''+\sqrt{2 w_\delta}=1,\quad w_\delta(\delta)=u_\delta,\ w_\delta'(\delta)=u'_\delta,
 \end{equation} 
 where $u_\delta = u(\delta),\ u'_\delta=u'(\delta)$ and $u$ is a solution of \eqref{eq:GovFinal}. Then we can estimate an error on the interval $[\delta,T]$ as follows
 \begin{equation}\label{ineq:IntervalLarge}
 |R(s)|=|u(s)-u_0(s)|\le |u(s)-w_\delta(s)|+|u_0(s)-w_\delta(s)|=|R_1|+|R_2|.
 \end{equation}
Now, we will analyze the error function $R_1$ and $R_2$ separately. First, let us insert $u=R_1+w_\delta$ into eq. \eqref{eq:GovFinal} to get
\begin{equation}
    R_1''+w_\delta''+\epsilon (R_1+w_\delta)+\sqrt{2(R_1+w_\delta)}=1,
\end{equation}
where $R_1$ is an unknown function. Next, we are using \eqref{eq:FunctionW} to introduce above equation in a more familiar form 
\begin{equation}\label{eq:ErrorFunct1}
    R_1''+\epsilon R_1'=1-\sqrt{2(R_1+w_\delta)}-\sqrt{2 w_\delta}-\epsilon w'_\delta,
\end{equation}
with the initial conditions $R_1(\delta)=u(\delta)-u_\delta=0$ and $R_1'(\delta)=u'(\delta)-u'_(\delta)=0$.
Subsequently, we integrate \eqref{eq:ErrorFunct1} twice and use a initial conditions of function $R_1$ to obtain
\begin{equation}\label{eq:R1Equation}
    |R_1|=\biggl|\int\limits_\delta^s G(s-x)f(w_\delta(x), R_1(x))\, d x\biggr|,
\end{equation}
where function $G(s-x)$ and $f(w_\delta, R_1)$ are defined like in eq. \eqref{eq:Remain}.  If we already have the integral equation, we can proceed to the further analysis of the eq. \eqref{eq:R1Equation} to get a appropriate integral inequality
 \begin{equation}\label{eq:R1Est}
 \begin{split}
 |R_1(s)|&\le \int_{\delta}^{s}|G(s-x)||f(w_\delta(x),R_1(x))|\,d x\le T \int_{\delta}^{s}|f(w_\delta(x),R_1(x))|\,d x\\
 &\le T \int_{\delta}^{s}(|\sqrt{2 w_\delta}-\sqrt{2(w_\delta+R_1)}|+\epsilon|w'_\delta|)\, d x\le T \int_{\delta}^{s}\frac{|R_1|}{\sqrt{2 w_\delta}+\sqrt{2(w_\delta+R_1)}}+\epsilon \, dx\\
 & \le  T \int_{\delta}^{s}\biggl(\frac{|R_1|}{\sqrt{2 u(\delta)}}+\epsilon\biggr) \, dx
 \end{split}
 \end{equation}

 Notice that for $s\in[\delta, T]$ we have $u(s)>0$. Using \eqref{ineq:SmallValues} we have
 \begin{equation}
 \frac{1}{3} \delta\le\sqrt{2 u(\delta)},
 \end{equation}
 and, therefore, we can estimate the denominator in (\ref{eq:R1Est}). Now, we define a new function
 \begin{equation}\label{eq:FunctionS}
 S(s):=T \int_{\delta}^{s}\biggl(\frac{3 |R_1|}{\delta}+\epsilon\biggr) \, dx.
 \end{equation}
 Then, of course
 \begin{equation}\label{eq:Inequality}
 |R_1(s)|\le S(s).
 \end{equation}
 Differentiating \eqref{eq:FunctionS} with respect to $s$ leads to
 \begin{equation}
 \frac{d}{d s}S(s)=T\biggl(\frac{3 |R_1|}{\delta}+\epsilon\biggr) \le T\biggl(\frac{3 S}{\delta}+\epsilon\biggr)
 \end{equation}
 Multiplying by $\exp{-(3 T s/\delta )} $ and integrating implies that
 	\begin{equation}
 	S(s)\le -\epsilon\frac{\delta}{3 T} \biggl(e^{-\frac{3 T}{\delta}s}-e^{-3 T}\biggr)e^{\frac{3 T}{\delta}s}\le \epsilon\frac{\delta}{3 T} \biggl(e^{-3 T+\frac{3 T}{\delta}s}-1\biggr)\le \epsilon\frac{\delta}{3 T} e^{-3 T+\frac{3 T^2}{\delta}},
 	\end{equation}
 where we have used the fact that $S(\delta)=0$. Now, if we choose $\delta(\epsilon):=-3 T^2/\log{\epsilon}$ we see that $\delta(\epsilon)\to 0$ and $\max_{s\in[\delta,T]}|S(s)|\to 0$ as $\epsilon\to 0$. Hence, the remainder $R_1$ vanishes. 
 	
 Now, we take a look at the second component on the right hand side in \eqref{ineq:IntervalLarge}, namely the remainder denoted by $R_2$. Let us notice that the functions $w_\delta$ and $u_0$ satisfy the same nonlinear differential equation but with different initial conditions at point $s=\delta$. Using the results concerning function $u_0$ presented at beginning of the section we obtain that $w_\delta$ is also a periodic function. Multiplying \eqref{eq:FunctionW} by $w_\delta'$ and integrating we get
 	\begin{equation}
 	   \frac{w_\delta'(s) ^2}{2}+\frac{2\sqrt{2}}{3}w_\delta(s)^{\frac{3}{2}}-w_\delta(s)-\frac{{u'}_{\delta} ^{2}}{2}-\frac{2\sqrt{2}}{3}u_\delta ^{\frac{3}{2}}+u_\delta=0,
 	\end{equation}
 	for $s\in [\delta, T]$. Using above equation we can show that the function $w_\delta$ is bounded. Specifically, following inequalities are satisfied
 	\begin{equation}
 	    \min_{s\in [\delta, T]}{\{u_0(s)\}}=0\le w_{\delta,min}\le w_\delta(s)\le w_{\delta, max}\le \frac{9}{8}=\max_{s\in [\delta, T]}{\{u_0(s)\}},
 	\end{equation}
 	for all $s\in [\delta, T]$ and all $\delta\ge 0$. Next, performing some technical and tedious calculations it can be shown that $\lim _{\delta\to 0^+} w_{\delta, min}=0$ and $\lim _{\delta\to 0^+} w_{\delta, max}=9/8$. We expect that the function $w_\delta$ converges to $u_0$ as $\delta\to 0^+$. In our previous work \cite{Plo18} we proved that \eqref{eq:GovFinal} posses a unique solution with zero initial conditions and with initial conditions of the form: $u(0)=\beta^2/2$ and $u'(0)=0$, where $\beta$ is some arbitrary number. The uniqueness of the solution can be easily extended to all initial conditions belonging to the set $\Omega:=\{(x_0, y_0): y_0^2/2+2\sqrt{2}/3 x_0^{3/2}-x_0\le 0\}$. Furthermore, the uniqueness theorem is true regardless of the value of the parameter $\epsilon$. Therefore, setting $\epsilon=0$ in \eqref{eq:GovFinal} implies that the solution of
 	\begin{equation}\label{eq:WithoutD}
 	        u''+\sqrt{2 u}=1,
 	\end{equation}
 	exists and is unique for all initial conditions from the set $\Omega$. Notice that the initial conditions of function $w_\delta$ belongs to the set $\Omega$. Setting $x:=u$ and $y:=u'$ we can transform eq. \eqref{eq:WithoutD} into a system of differential equations
 	\begin{equation}\label{System}
 	    \left\{\begin{array}{cl}
 	       x'  =& y , \\
 	        y' =& 1-\sqrt{2 x} .
 	    \end{array}
 	    \right.
 	\end{equation}
 	The uniqueness of the solution of \eqref{eq:WithoutD} for all initial condition belonging to the set $\Omega$ implies that all trajectories of the system \eqref{System} located in the set $\Omega$ are uniquely determined. Let $\Gamma_{w_\delta} :=\{(w_\delta(s),w'_\delta(s)),\ \delta\le s\le T\} $ and $\Gamma_{u_0}:=\{(u_0(t),u'_0(t)),\ \delta\le s\le T\} $ denote the trajectories associated with $w_\delta$ and $u_0$ respectively. After a detailed analysis of $w_\delta$ and $u_0$, we get  $\Gamma_{u_0}\in \partial \Omega$ and $\Gamma_{w_\delta}\in\Omega$ for all $\delta \ge 0$. If a value of parameter $\delta$ will decrease to zero we anticipate that the elements of $\Gamma_{w_\delta}$ will be approaching the proper elements in $\Gamma_{u_0}$. Using the fact that the minimal and maximal value of $w_\delta$ converge to $0$ and $9/8$ respectively we get

 	\begin{align}
 	    \lim\limits_{\delta\to 0^+}(w_{\delta,\min},0)& = (0,0)=(u_{0,\min},0),\\
 	    \lim\limits_{\delta\to 0^+}(w_{\delta,\max},0)& = \biggl(\frac{9}{8},0\biggr)=(u_{0,\max},0).
 	\end{align}
 	Now, applying a non-intersection property for uniquely defined phase plane of system \eqref{System} we get $\lim_{\delta\to 0^+}\Gamma_{w_\delta} =\Gamma_{u_0}$. Therefore, combining that with the fact that the difference between function $w_\delta$ and $u_0$ at initial point converge to zero as $\delta\to 0^+$ we get $\lim_{\delta\to 0^+}\max_{s\in [\delta, T]}||(w_\delta(s),w'_\delta(s))-(u_0(s),u'_0(s))||=0$ what immediately  implies $\lim_{\delta\to 0^+}|R_2|=\lim_{\delta\to 0^+}|w_\delta -u_0|=0$. Here the $||\cdot||$ is a Euclidean norm in $\mathbb{R}^2$

 	Finally, we take togather the inegualities \eqref{ineq:IntervalSmall} and \eqref{ineq:IntervalLarge} concerning the bound for remainder on the intervals $[0,\, \delta]$ and $[\delta,\, T]$ respectively and using $\lim_{\delta\to 0^+}|w_\delta -u_0|=0$  we conclude the final convergence result \eqref{TheoremResult}.
 	
\end{proof}

The exact order of convergence in Theorem \ref{thm:Error} is difficult to obtain because of the appearance of the non-Lipschitz function in the main equation. However, the numerical analysis shows that the error rate between $w_\delta$ and $u_0$ should not be worse than $\delta$ as $\delta\to 0^+$. In the following remark we used the MATHEMATICA software to find a numerical solution to eq. \eqref{eq:UndampedU} with the non zero initial conditions. 
\begin{rem}
The numerical analysis revealed that the following convergence error holds
\begin{equation}\label{eq:ConvergenceError}
    \max\limits_{s\in[0,T]}|u(s)-u_0(s)|=\mathcal{O}\biggl(-\frac{1}{\log{\epsilon}}\biggr)\quad \textrm{as}\quad \epsilon\to 0^+,
\end{equation}
where $u(s)$ is an exact solution of \eqref{eq:GovFinal}, function $u_0$ is defined as before and $T$ is an any fixed constant greater than zero.
\end{rem}
From the proof of Theorem \ref{thm:Error} we have that on the interval $[0,\delta]$ the error function between $u(s)$ and $u_0$ can be bounded from above by
\begin{equation}
    |R(s)|\le \biggl(\frac{3}{2}+\epsilon\biggr)\delta^2.
\end{equation}
Whereas on the interval $[\delta, T]$ we get
\begin{equation}
    |R(s)|\le \frac{\delta}{3 T} e^{-3 T}+\max\limits_{s\in [\delta, T]}|w_\delta(s)-u_0(s)|.
\end{equation}
We numerically compute the values of the second term in above inequality  for various values of $\delta$. The result are presented on Fig. \ref{fig:NumericalError}. In calculation we use the asymptotic results from Lemma \ref{Lemma1}, specifically we use following initial conditions for function $w_\delta(s)$ on the interval $[\delta, T]$
\begin{equation}
w_\delta(\delta)=D_1 \delta^2,\quad w_\delta(\delta)=D_2\delta,
\end{equation}
where $D_1$ and $D_2$ are some positive constants.
The numerical results showed that the error between function $w_\delta$ and $u_0$ might of order of $\delta$. It is worth to emphasize that for some choice of initial conditions $w_\delta$ the quadratic order of convergence is observed. 
\begin{figure}
	\centering
	\includegraphics[scale=0.7]{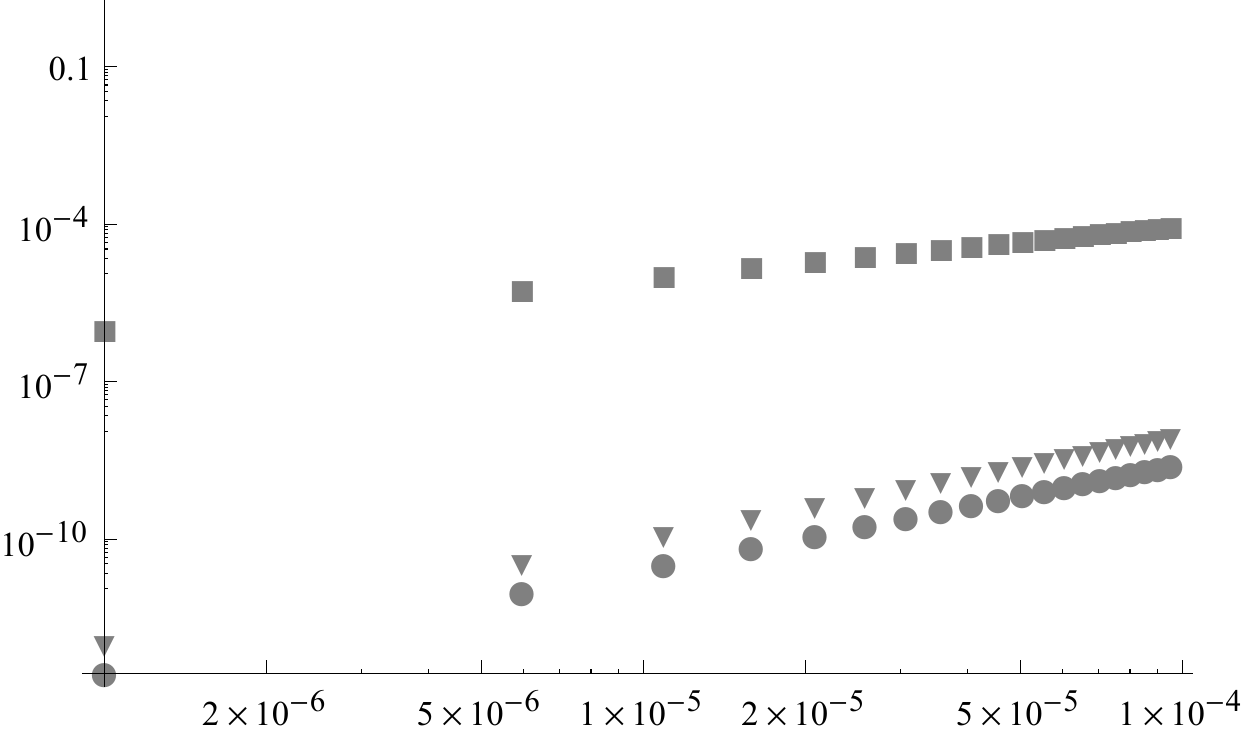}
	\includegraphics[scale=0.7]{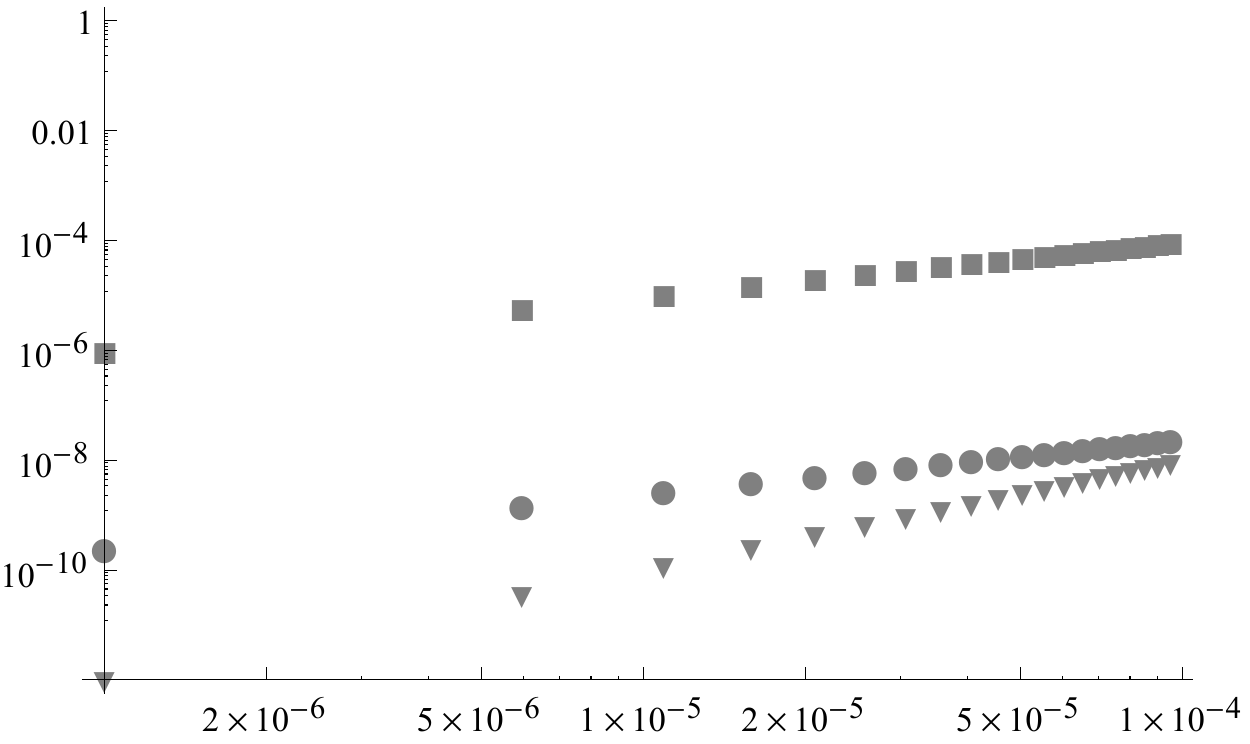}
	\caption{The numerical error between function $w_\delta$ and $u_0$ ($\max_{s\in[\delta, T]}|w_\delta-u_0|,\greycircle$) for various values of $\delta$ and with different sets of initial conditions: $w_\delta(\delta)=1/2 \,\delta^2$, $w'_\delta(\delta)=\delta$ (left) and $w_\delta(\delta)=7/18 \,\delta^2$, $w'_\delta(\delta)=1/18\,\delta$ (right). The functions $g_1(\delta)=\delta\ (\greysquare)$ and $g_2(\delta)=\delta^2\ (\greytriangle)$ were marked on the graphs for comparison.}
	\label{fig:NumericalError}
\end{figure}
Combing numerical simulations with the above results we conclude that \eqref{eq:ConvergenceError} holds.

\section{Asymptotic behavior for large times}\label{SectionAsym}

In this section we consider a large time behaviour of the solution of \eqref{eq:GovFinal} in the oscillatory regime. To facilitate the analysis it is useful to introduce an additional function $\phi$ that will transform \eqref{eq:GovFinal} into a simpler form. We set
\begin{equation}\label{DefLiouvilleGreen}
    v(s):=(u(s)-1/2)\phi(s)^{-1},
\end{equation}
where the exact formula for $\phi(s)$ will be found later. Let us notice that \eqref{DefLiouvilleGreen} is simply the Liouville transformation. It is easy to see that $u(s)=1/2+\phi(s) v(s)$ and after substituting the above into \eqref{eq:GovFinal} we obtain
\begin{equation}
    \phi v''+(2 \phi' +\epsilon \phi) v'+(\phi''+\epsilon \phi')v+\sqrt{1+2 \phi v}=1.
\end{equation}
We can see that in order to eliminate the first derivative we have to take
\begin{equation}
    \phi(s):=e^{\epsilon s/2}
\end{equation}
Now we can rewrite \eqref{eq:GovFinal} as
\begin{equation}\label{eq:Undamped}
    v''+\left(\frac{2}{1+\sqrt{1+2 \phi v}}-\frac{\epsilon^2}{4}\right)v=0,\quad v(0)=-\frac{1}{2},\ v'(0)=\frac{\epsilon}{4}.
\end{equation}
We can think about the above equation as one describing nonlinear oscillations in which the frequency is modulated with the amplitude via the term in the parentheses. 

We will use integral equations techniques to investigate the large time regime. Let us consider a time interval $[T, T^*]$, where $T$ and $T^*$ are fixed numbers. With a standard use of the Green's function we can transform \eqref{eq:Undamped} into the integral form
\begin{equation}\label{eq:IntegralForm}
    v(s)=A_0\cos{\tau s}+ B_0\sin{\tau s}+\int\limits_{T}^{s}G(s-t) f(\phi(t),v(t))\, d t=:L v,
\end{equation}
where
\begin{flalign*}
    G(z)& = \frac{1}{\tau}\sin{\tau z},\qquad \tau=\sqrt{1-\frac{\epsilon^2}{4}},\\
    f(\phi(t),v(t)) & = \biggl(1-\frac{2}{1+\sqrt{1+2 \phi(t) v(t)}}\biggr)v(t),
\end{flalign*} 
and 
\begin{align}
    A_0 = &\, \frac{v'(T
    ) \cos (T \tau )}{\tau }+v(T) \sin (T \tau ),\\
    B_0 = &\, v(T) \cos (T \tau )-\frac{v'(T) \sin (T \tau )}{\tau },
\end{align}
The initial values $v_T:=v(T),\ v'_T:=v'(T)$ are well defined because of our preceding results stating that \eqref{eq:GovFinal} has a unique global solution (see \cite{Plo18}). 

Immediately, we can make some initial observations about the function $v$.
\begin{cor}\label{cor:Bound}
	Let $v$ be a solution of \eqref{eq:Undamped} and assume that $\epsilon > 0$. Then, for sufficiently large $s>0$ we have
	\begin{equation}
	-\frac{1}{2}<C_1\le \phi(s)v(s)\le C_2< \frac{1}{2},
	\end{equation}
	where $C_{1,2}$ are some fixed constants. 
\end{cor}  
\begin{proof}
From our previous result \cite{Plo18} we know that $0\le u\le 9/8$ where $u$ is a solution of \eqref{eq:GovFinal}. We will show that the there exist such $T$ that the inequality $0<1/2+C_1\le u$ is satisfied on the interval $[T,\infty)$. To do this let rewrite the energy form of the governing equation \eqref{eq:GovFinal}
	\begin{equation}
	\frac{u'^2}{2}+\frac{2\sqrt{2}}{3}u^{\frac{3}{2}}-u=-\epsilon \int\limits_{0}^{s}u'^2\,d t.
	\end{equation}
	Let assume that $u$ reaches zero at some point $s^*>0$. Then, the above immediately leads to a contradiction since the right-hand side is strictly negative. Hence the only point in which function $u$ can reach zero is $s=0$.
	
	Moreover, from our previous work \cite{Plo18} we know that the solution of \eqref{eq:GovFinal} oscillates around $1/2$ and approaches it as $s\rightarrow\infty$. Hence, we are able to choose such $T_1$ that the inequality $0<1/2+C_1\le u$ is satisfied for $s\in[T_1,\infty)$. Using the transformation (\ref{DefLiouvilleGreen}) yields the estimate from below. Similarly, on the grounds of the asymptotic convergence of function $\phi v$ there must exist $T_2$ such that the inequality $\phi v\le C_2<1/2$ is satisfied for $s\in[T_2,\infty)$. Finally, we put $T:=\max\{T_1,T_2\}$.
\end{proof}
Now, we are able to introduce result concerning existence and uniqueness of solutions to (\ref{eq:IntegralForm}). Later, we will use it to construct several large time approximations.
\begin{thm}
Solution of \eqref{eq:IntegralForm} exists and is uniquely defined on $[T,\infty)$, where $T$ is a fixed number defined as in the proof of Corollary \eqref{cor:Bound} .

\end{thm}

\begin{proof}
	Let us consider the existence problem of the solution of \eqref{eq:IntegralForm} on the interval $[T,T^*]$, where $T^*$ is an arbitrary fixed number greater than $T$. Moreover we are looking for the solution of \eqref{eq:IntegralForm} in the Banach space $(C([T,T^*]),\left\|\cdot\right\|_\infty$). Further, introduce the Bielecki's norm defined as follows \cite{bielecki56}
	\begin{equation}\label{BieleckiNorm}
	\left\|u\right\|_B:=\sup\limits_{s\in[T,T^*]}|e^{-\alpha s}u(s)|,\quad\alpha>0,
	\end{equation}
	where $\alpha$ is a certain fixed number. Because of the the fact that Bielecki's norm is equivalent to the usual supremum norm on $[T,T^*]$ we can use it in the proof that the operator $A$ is a contraction with respect to the supremum norm.
	
	For any continuous $u$ and $v$ we have
	\begin{equation}\label{ineq:FunctionsDifference}
	\begin{split}
	|f(\phi(t), u(t)) -f(\phi(t),v(t) )|\le & \biggl(1+\frac{2}{\sqrt{1+2 \phi u}+\sqrt{1+2 \phi v}}\biggr)|u-v| \\
	\le & \biggl(1+\frac{1}{\sqrt{1+2 C_1}}\biggr)|u-v|,
	\end{split}
	\end{equation}
	where the last inequality is a consequence of Corollary \ref{cor:Bound} with appropriate chosen value of $T$. Finally, using \eqref{eq:IntegralForm} we can make the estimate
	\begin{equation}
	\begin{split}
	|L(u)(s)-L(v)(s)|\le &\int\limits_{T}^{s}|G(s-t)| |f(\phi(t),u(t))-f(\phi(t),v(t))|\, d t\le \frac{1}{\tau}\int\limits_{T}^{s}|f(\phi(t),u(t))-f(\phi(t),v(t))|\, d t\\
	\le & \frac{1}{\tau}\int\limits_{T}^{s}e^{\alpha t}e^{-\alpha t}|f(\phi(t),u(t))-f(\phi(t),v(t))|\, d t\le \frac{e^{\alpha s}-e^{\alpha T}}{\tau \alpha} \left\|f(\phi,u)-f(\phi,v)\right\|_B . 
	\end{split}
	\end{equation} 
	Multiplying the above by $\exp(-\alpha s)$, taking supremum and using \eqref{ineq:FunctionsDifference} leads to
	\begin{equation}
	\left\|L(u)-L(v)\right\|_B\le \frac{1}{\tau \alpha} \left(1+\frac{1}{\sqrt{1+2 C_1}}\right)\left\|u-v\right\|_B.
	\end{equation}
	Choosing a sufficiently large $\alpha$ we can make the prefactor on the right-hand side of the above to be smaller than $1$ which shows that the operator $L$ is a contraction in the Bielecki's norm. The equivalence of the supremum and Bielecki's norms implies that $L$ is also a contraction in the supremum norm. Banach Contraction Mapping Theorem used on the space $(C([T,T^*]),\left\|\cdot\right\|_\infty)$ asserts that \eqref{eq:IntegralForm} posses a unique fixed point. From the arbitrariness of $T^*$ we conclude that the solution can be continued to the unbounded domain $[T,\infty)$.
\end{proof}

The above result suggests that we may be able to explicitly construct the exact solution of the problem (\ref{eq:IntegralForm}) as a limit of iterative sequence. This can be utilized in finding the leading order asymptotic behaviour for large times and provide some useful approximations.

To this end, let us notice that the function $f$ in \eqref{eq:Undamped} can be expanded in a Taylor series with respect to the product $\phi(s) v(s)$. The function $\phi(s)$ decreases monotonically to zero while the function $v$ is bounded. Therefore, for $s\to\infty$ we have $\phi(s)v(s)\to 0$. Hence, we can rewrite \eqref{eq:IntegralForm} in the following form
\begin{equation}\label{eq:IntegralUndamped}
 v(s)=A_0\cos{\tau s}+B_0\sin{\tau s}-\sum\limits_{m=1}^{\infty}\sigma_m\int_{T}^{s}G(s-t) \phi(t)^m v^{m+1}\, d t=:P v, 
\end{equation}
where $A_0$ and $B_0$ are defined as before and the coefficients are
\begin{equation}
\sigma_m=\binom{\frac{1}{2}}{m+1}2^{m+1}.
\end{equation}
Equation \eqref{eq:IntegralUndamped} can now be used to construct a iteration
\begin{equation}\label{eq:IntegralUndampedIteration}
    v_{n+1}=P v_n,\qquad v_0=A_0 \cos{\tau s}+B_0 \sin{\tau s}.
\end{equation}
\begin{thm}
Let $v^*$ be the unique solution of \eqref{eq:IntegralUndamped}. Then $\{v_n\}$ is convergent to unique $v^*$. Moreover, 
\begin{equation}\label{AsymBehavior}
    v_{n+1}(s)\sim A_n \sin{\tau s}+B_n \cos{\tau s},\quad \textrm{as}\ s\to \infty,
\end{equation}
with
\begin{equation}\label{AsymBehaviorCoeff}
    \begin{split}
        A_{n+1}=& \frac{v'(T
    ) \cos (T \tau )}{\tau }+v(T) \sin (T \tau )-\frac{1}{\tau}\sum\limits_{m=1}^{\infty}\, \sum\limits_{\substack{j,k\\ j+k=m+1}}\binom{m+1}{j}\sigma_m A_{n}^{j}B_{n}^{m+1-j}J_1,\\
        B_{n+1}=& v(T) \cos (T \tau )-\frac{v'(T) \sin (T \tau )}{\tau }+\frac{1}{\tau}\sum\limits_{m=1}^{\infty}\sum\limits_{\substack{j,k\\ j+k=m+1}}\binom{m+1}{j}\sigma_m A_{n}^{j}B_{n}^{m+1-j}J_2,
    \end{split}
\end{equation}
where
\begin{equation}\label{CoefIntegral1}
    \begin{split}
        J_1=& -\frac{1}{2^{j+k+1}}\sum\limits_{l=0}^{j}\sum\limits_{p=0}^{k+1}\binom{j}{l}\binom{k+1}{p}\frac{f(j,l,k,p,m,\tau,\epsilon)}{(\tau(j-2 l+k+1-2 p))^2+\lambda^2},\\
        f(j,l,k,p,m,\tau,\epsilon)=&\left\{\begin{array}{ll}
           (-1)^{r+l+1}(\lambda \cos{\nu T}-\nu \sin{\nu T})e^{-\lambda T},  &\mathrm{for}\, j=2 r,  \\
             (-1)^{r+l}(\nu \sin{\nu T}-\lambda \cos{\nu T})e^{-\lambda T}, & \mathrm{for}\, j=2 r+1,
        \end{array}\right.
    \end{split}
\end{equation}
and
\begin{equation}
    \begin{split}
        J_2=& -\frac{1}{2^{j+k+1}}\sum\limits_{l=0}^{j+1}\sum\limits_{p=0}^{k}\binom{j+1}{l}\binom{k}{p}\frac{g(j,l,k,p,m,\tau,\epsilon)}{(\tau(j-2 l+k+1-2 p))^2+\lambda^2},\\
        g(j,l,k,p,m,\tau,\epsilon)=&\left\{\begin{array}{ll}
           (-1)^{r+l}(\lambda \cos{\nu T}-\nu \sin{\nu T})e^{-\lambda T},  &\mathrm{for}\, j=2 r+1,  \\
             (-1)^{r+1+l}(\nu \cos{\nu T}-\lambda \sin{\nu T})e^{-\lambda T}, & \mathrm{for}\, j=2 r,
        \end{array}\right.
    \end{split}
\end{equation}
and $\nu=\tau (j-2 l+k+1-2 p)$.
\end{thm}
\begin{proof}
    Let us prove that the iteration
    \begin{equation}
        v_{n+1}=P v_n,\quad v_0=A_0 \sin{\tau s}+B_0 \cos{\tau s},
    \end{equation}
    converges to the exact solution of \eqref{eq:IntegralForm} on the interval $[T,T^*]$, where $T^*$ is an arbitrary fixed number greater than $T$. Take any two solutions of \eqref{eq:IntegralForm}, call them $u$ and $v$ and calculate
    \begin{equation}\label{ineq:SeriesContraction}
        |P v-P u|\le \sum\limits_{m=1}^{\infty}\sigma_m\int_{T}^{s}|G(s-t)|\phi(t)^{-1}|(\phi v)^{m+1}-(\phi u)^{m+1}|\,d t.
        \end{equation}
        Then, we can apply Lagrange's Mean Value Theorem for the function $(\phi v)^{m+1}$ for all values of $t$. Thus, we have
        \begin{equation}
            \begin{split}
\sum\limits_{m=1}^{\infty}\sigma_m\int_{T}^{s}|G(s-t)|\phi(t)^{-1}|(\phi v)^{m+1}-(\phi u)^{m+1}|\,d t\le &\sum\limits_{m=1}^{\infty}\sigma_m(m+1)\int_{T}^{s}|G(s-t)|w(t)^m|v-u|\,d t,
\end{split}
\end{equation}
where $w$ is a some function satisfying $\min\{\phi(t) v(t),\phi(t) u(t)\} <w(t) < \max\{\phi(t) v(t),\phi(t) u(t)\}$ for all $t\in [T,s]$. From the Cor. \ref{cor:Bound} we have that $\phi(s) v(s)\le C_2<1/2$, where $v$ is a solution of \eqref{eq:IntegralForm}. Hence we can finally write
\begin{equation}
    \begin{split}
        |P v-P u|\le&\sum\limits_{m=1}^{\infty}\sigma_m(m+1)C_2^m\int_{T}^{s}e^{\alpha t}|G(s-t)|e^{-\alpha   t}|v-u|\,d t\\
        \le & \left\|v-u\right\|_B\sum\limits_{m=1}^{\infty}\sigma_m(m+1)C_2^m\int_{T}^{s}e^{\alpha t}|G(s-t)|\,d t\\
        \le& \left\|v-u\right\|_B\sum\limits_{m=1}^{\infty}\sigma_m(m+1)C_2^m\frac{1}{\tau}\int_{T}^{s}e^{\alpha t}\,d t\\
        =&e^{\alpha s}\left\|v-u\right\|_B\sum\limits_{m=1}^{\infty}\sigma_m(m+1)C^m\frac{1}{\tau \alpha}\biggl(1-e^{-\alpha (s- T)}\biggr)\\
        \le & e^{\alpha s}\left\|v-u\right\|_B\sum\limits_{m=1}^{\infty}\sigma_m(m+1)C_2^m\frac{1}{\tau \alpha},
        \end{split}
    \end{equation}
    where $||\cdot||_B$ is a Bielecki's norm defined as in \eqref{BieleckiNorm} and $C_2$ is a fixed number smaller than $1/2$ introduced earlier in \eqref{cor:Bound}.

     From the standard convergence criterion we know that the series in above converge. Moreover, we have
    \begin{equation}
    \begin{split}
       \sum\limits_{m=1}^{\infty}\frac{-\sigma_m(m+1)C_2^m }{\tau \alpha} =\frac{-C_2}{2\alpha \tau}\biggl(\, _2F_1(1,1.5;3;-2 C_2)- \, _2F_1(1.5,2;3;-2 C_2)\biggr)=:H(\alpha, T),
       \end{split}
    \end{equation}
    where $_2F_1$ is a hypergeometric function
    \begin{equation}
        \, _2 F_1 (a,b,c,z)=\sum\limits_{k=0}^{\infty}\frac{(a)_k (b)_k}{(c)_k}\frac{z^k}{k!},
    \end{equation}
    with
    \begin{equation}
        (a)_k=\left\{\begin{array}{ll}
            1 & \textrm{for } k=0, \\
            a(a+1)\cdots (a+n-1) &\textrm{for } k>0. 
        \end{array}\right.
    \end{equation}
    Now we are able to rewrite the main inequality of \eqref{ineq:SeriesContraction} in the following form
    \begin{equation}
        \left\|P(u)-P(v)\right\|_B\le  |H(\alpha,T)| \left\|u-v\right\|_B.
    \end{equation}
    Let us notice the inverse dependence of function $H(\alpha, T)$ on the parameter $\alpha$. Hence, choosing sufficiently large  $\alpha$ we can make expression in brackets smaller than $1$ and whence show that $P$ is a contraction in Bielecki's metric on considered interval. 
    
    From the  arbitrariness of $T^*$ we conclude that the operator $P$ is a contraction on the whole interval $[0,\infty]$.

From the numerical analysis and observations for large arguments it is justified to approximate function $v$ by some trigonometric function. Therefore, the natural candidate for the initial guess is $v_0$ satisfying $L v_0=0$, with appropriate initial conditions. Hence,
    \begin{equation} \label{def:V0}
        v_0(s)=B_0\cos{\tau s}+A_0\sin{\tau s}.
    \end{equation}
    Next, we can expand the kernel $G(z)$ in the integral in \eqref{eq:IntegralUndamped} and get
    \begin{equation}\label{eq:IterationProof}
    \begin{split}
     v(s)= &\biggl(B_0+\frac{1}{\tau}\sum\limits_{m=1}^{\infty}\sigma_m\int_{T}^{s}\sin{(\tau t)}\, \phi(t)^m v(t)^{m+1}\, d t\biggr)\cos{\tau s}\\
     &+\biggl(A_0-\frac{1}{\tau}\sum\limits_{m=1}^{\infty}\sigma_m\int_{T}^{s}\cos{(\tau t)}\, \phi(t)^m v(t)^{m+1}\, d t\biggr)\sin{\tau s}.  
     \end{split}
    \end{equation}
     It is easy notice that if a $v_0$ is a sum of trigonometric functions then the subsequent result of iteration will be also in the similar form. In the another words, function $v_n$ will contain only a linear combination of functions $\sin{\tau s}$ and $\cos{\tau s}$. We will use the mathematical induction to prove that the function $v_n$ can be reduced to the form 
    \begin{equation}\label{AsymExpressionV}
        v_n(s)\sim A_n \sin{\tau s}+B_n\cos{\tau s},\quad \textrm{as}\ s\to \infty,
    \end{equation}
    for all $n\in \mathbb{N}$. As will be shown the $A_n$ and $B_n$ coefficients in fact dependent only on the values of $A_{n-1}$, $B_{n-1}$ and on initial coefficients $A_0$ and $B_0$.

    The base step of mathematical induction of this  issue is crucial. From \eqref{def:V0} we have that the function $v_0$ is a linear sum of sine and cosine functions. If we substitute the $v_0$ to \eqref{eq:IterationProof} and use a binomial expansion we get
    \begin{equation}
     \begin{split}
        v_{1}=&\biggl(A_0-\frac{1}{\tau}\sum\limits_{m=1}^{\infty}\, \sum\limits_{\substack{j,k\\ j+k=m+1}}\binom{m+1}{j}\sigma_m A_{0}^{j}B_{0}^{m+1-j}J_1\biggr)\sin{\tau s}\\
        &+ \biggl(B_0+\frac{1}{\tau}\sum\limits_{m=1}^{\infty}\sum\limits_{\substack{j,k\\ j+k=m+1}}\binom{m+1}{j}\sigma_m A_{0}^{j}B_{0}^{m+1-j}J_2\biggr)\cos{\tau s},
    \end{split}  
    \end{equation}
    where 
    \begin{align}
        J_1& = \int_{T}^{s}e^{-\lambda t} \sin{(\tau t)}^j \cos{(\tau s)}^{k+1}\, d t,\\
        J_2& =  \int_{T}^{s}e^{-\lambda t} \sin{(\tau t)}^{j+1} \cos{(\tau s)}^{k}\, d t.
    \end{align}
    Now, we will compute the values of above integrals to get the analytic form of the function $v_n(s)$ for large values of $s$. Let us first consider integral $J_1$
    \begin{equation}
        \begin{split}
         J_1& = \int_{T}^{s}e^{-\lambda t} \sin{(\tau t)}^j \cos{(\tau s)}^{k+1}\, d t=\int_{0}^{s}e^{-\lambda t}\biggl(\frac{e^{\tau t i}-e^{-\tau t i}}{2 i}\biggr)^{j}\biggl(\frac{e^{\tau t i}+e^{-\tau t i}}{2 }\biggr)^{k+1}\, d t\\ 
         & = \int_{T}^{s}e^{-\lambda t}\sum\limits_{l=0}^{j}\sum\limits_{p=0}^{k+1}\binom{j}{l}\binom{k+1}{p}\frac{1}{2^{j+k+1}i^j}(-1)^l e^{\tau s i(j-2 l+k+1-2 p)}\,d t\\
         &= \sum\limits_{l=0}^{j}\sum\limits_{p=0}^{k+1}\binom{j}{l}\binom{k+1}{p}\frac{1}{2^{j+k+1}i^j}(-1)^l \int_{T}^{s}e^{-\lambda t+\tau t i(j-2 l+k+1-2 p)}\,d t\\
         &= \sum\limits_{l=0}^{j}\sum\limits_{p=0}^{k+1}\binom{j}{l}\binom{k+1}{p}\frac{1}{2^{j+k+1}i^j} \frac{(-1)^l}{\tau i(j-2 l+k+1-2 p)-\lambda}e^{-\lambda t+\tau t i(j-2 l+k+1-2 p)}\biggl|_{t=T}^{t=s}\\
         &\sim -\frac{1}{2^{j+k+1}}\sum\limits_{l=0}^{j}\sum\limits_{p=0}^{k+1}\binom{j}{l}\binom{k+1}{p} \frac{(-1)^{l+j}(-\lambda-i\tau(j-2 l+k+1-2 p))i^j e^{-\lambda T+\mu T i}}{(\tau (j-2 l+k+1-2 p))^2+\lambda^2},\quad \textrm{as}\quad s\to\infty.
        \end{split}
    \end{equation}
    The imaginary parts of above expressions have to sum up to zero. Using the Euler's formula we expanded the real trigonometric functions and we removed only the vanishing terms at the upper limit of integration which does not affect on the realness of the final function. Hence we get the final expression \eqref{CoefIntegral1} of $J_1$ for $v_1$. The above  steps can be repeated as well for $J_2$ to get final asymptotic form of that integral.
    
    As we have seen the asymptotic expression for function $v_1$ is a linear sum of sine and cosine functions, where the coefficients are dependent only on $A_0$ and $B_0$. Moreover, if we now assume that the asymptotic behavior of function $v_n$ for some $n$ is expressed by the formula \eqref{AsymExpressionV} then repeating the tedious calculations performed above for function $v_1$ we get 
    \begin{equation}
        v_{n+1}\sim A_{n+1} \sin{\tau s} +B_{n+1} \cos{\tau s}.
    \end{equation}
    Since both the base case and the inductive step have been performed,  by mathematical induction we conclude the asymptotic form of function $v_n$ \eqref{AsymExpressionV} for arbitrary  $n\in \mathbb{N}$.
\end{proof}
\begin{figure}
	\centering
	\includegraphics[scale=0.45]{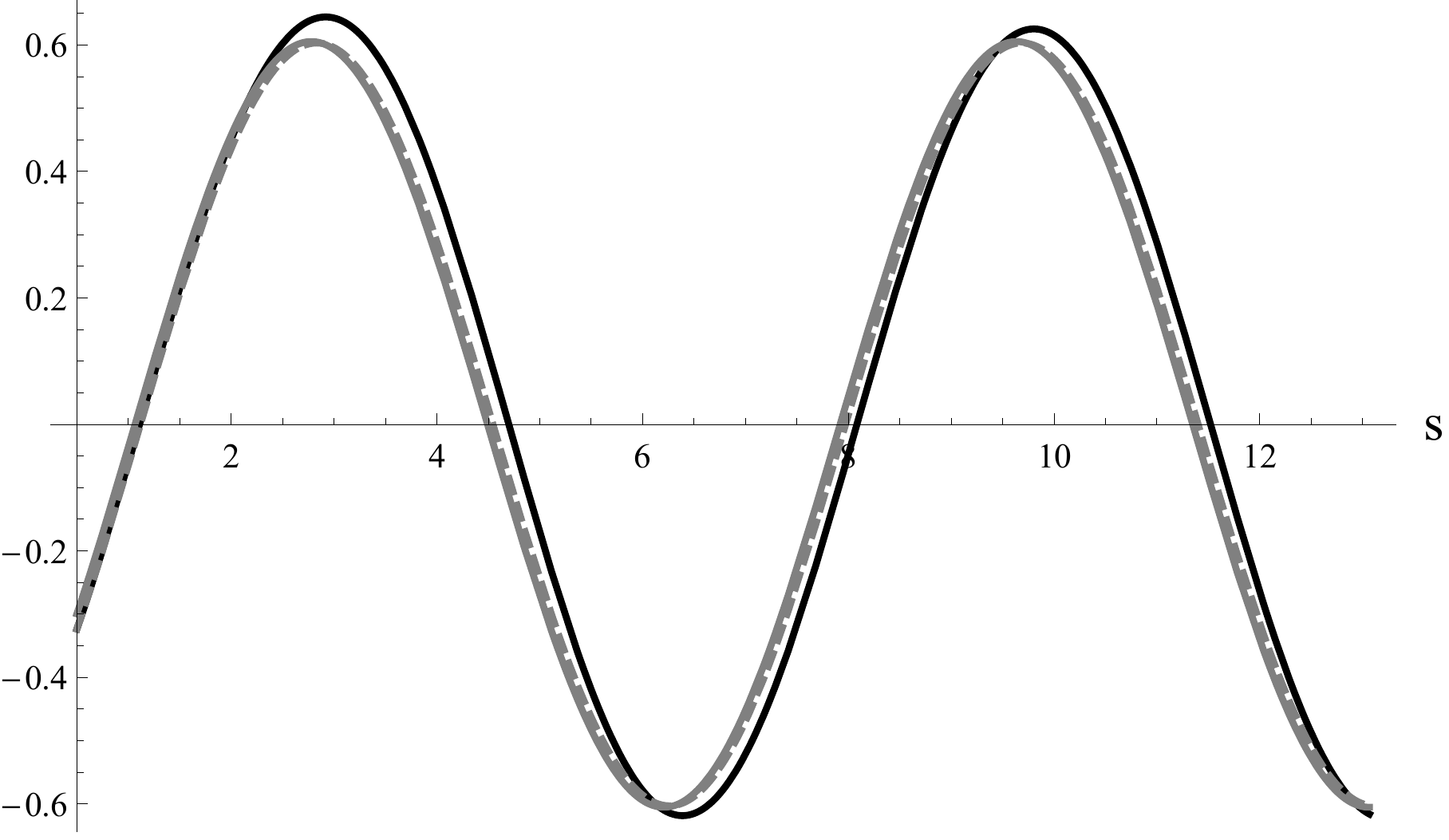}\hspace{2em}
	\includegraphics[scale=0.45]{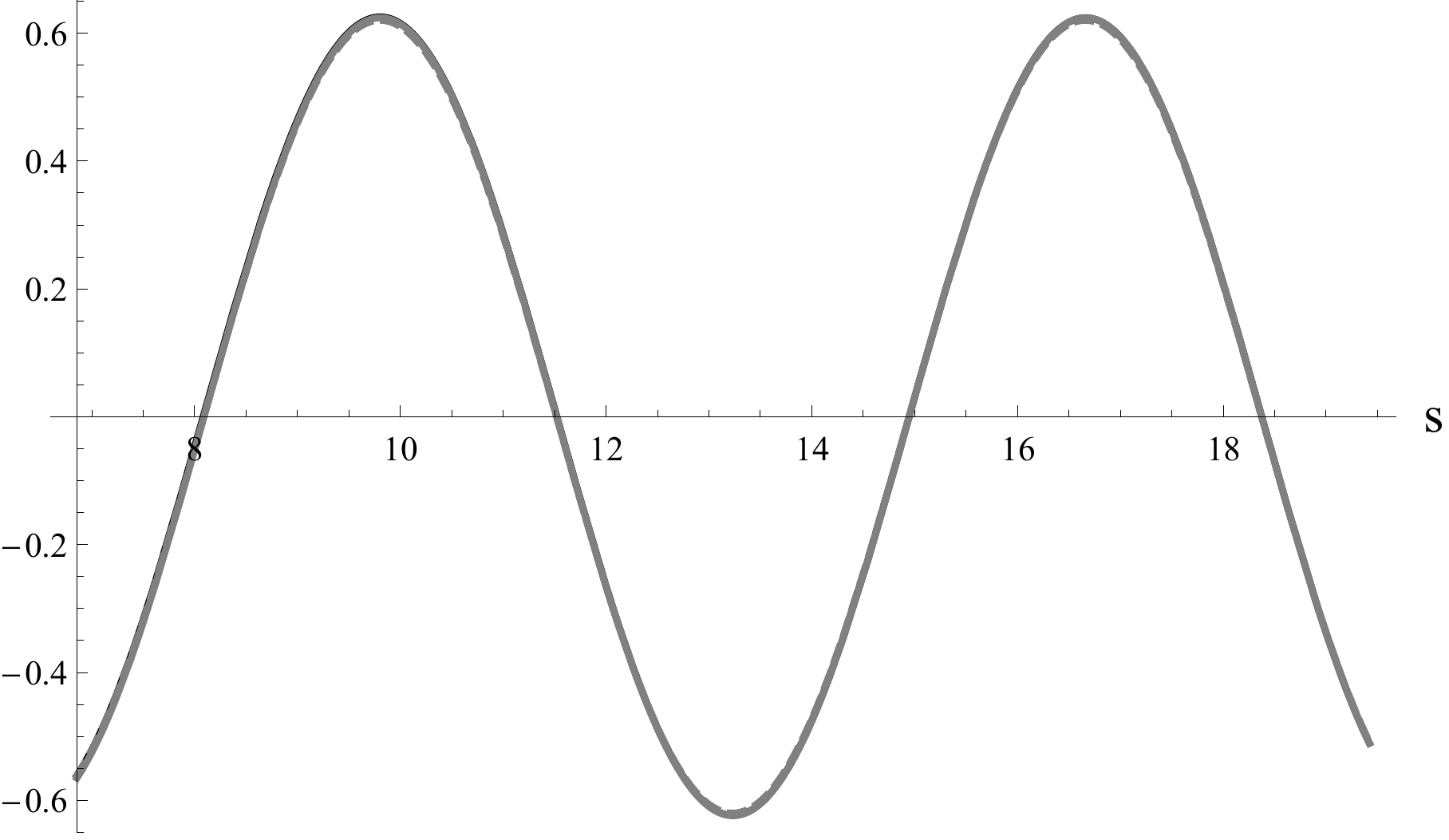}
	\caption{The asymptotic solutions $A_0 \sin{\tau s}+B_0\cos{\tau s}$ (dashed gray line) and  $A_6 \sin{\tau s}+B_6\cos{\tau s}$ (gray line) with  $T=0.5$ (left) and $T=6.855$ (right). The solid black line represents the numerical solution of \eqref{eq:IntegralForm} with $\epsilon=0.8$ .}
	\vspace{2em}
	\label{fig:AsymBehaviorLargeE}
	\includegraphics[scale=0.45]{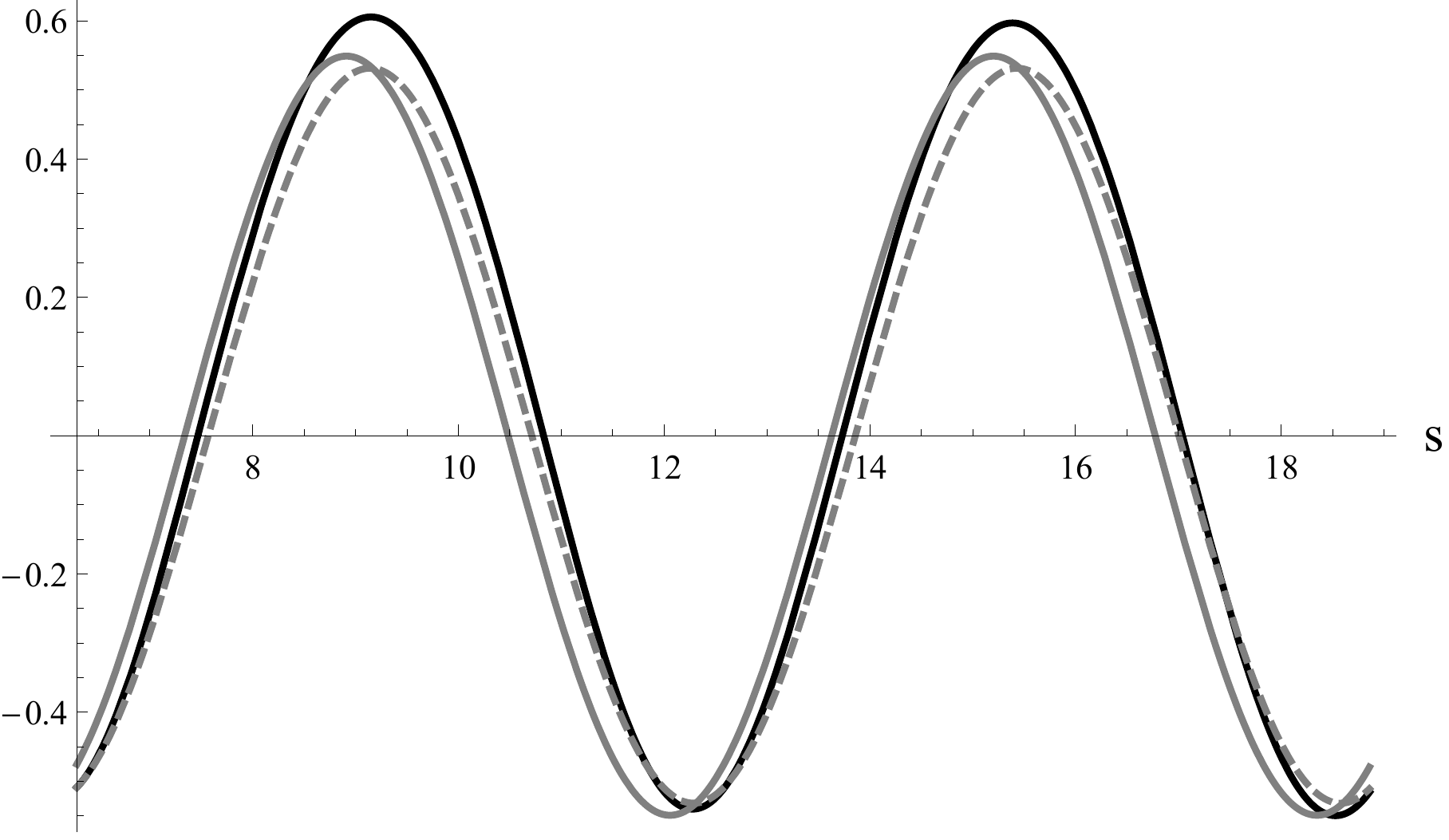}\hspace{2em}
	\includegraphics[scale=0.45]{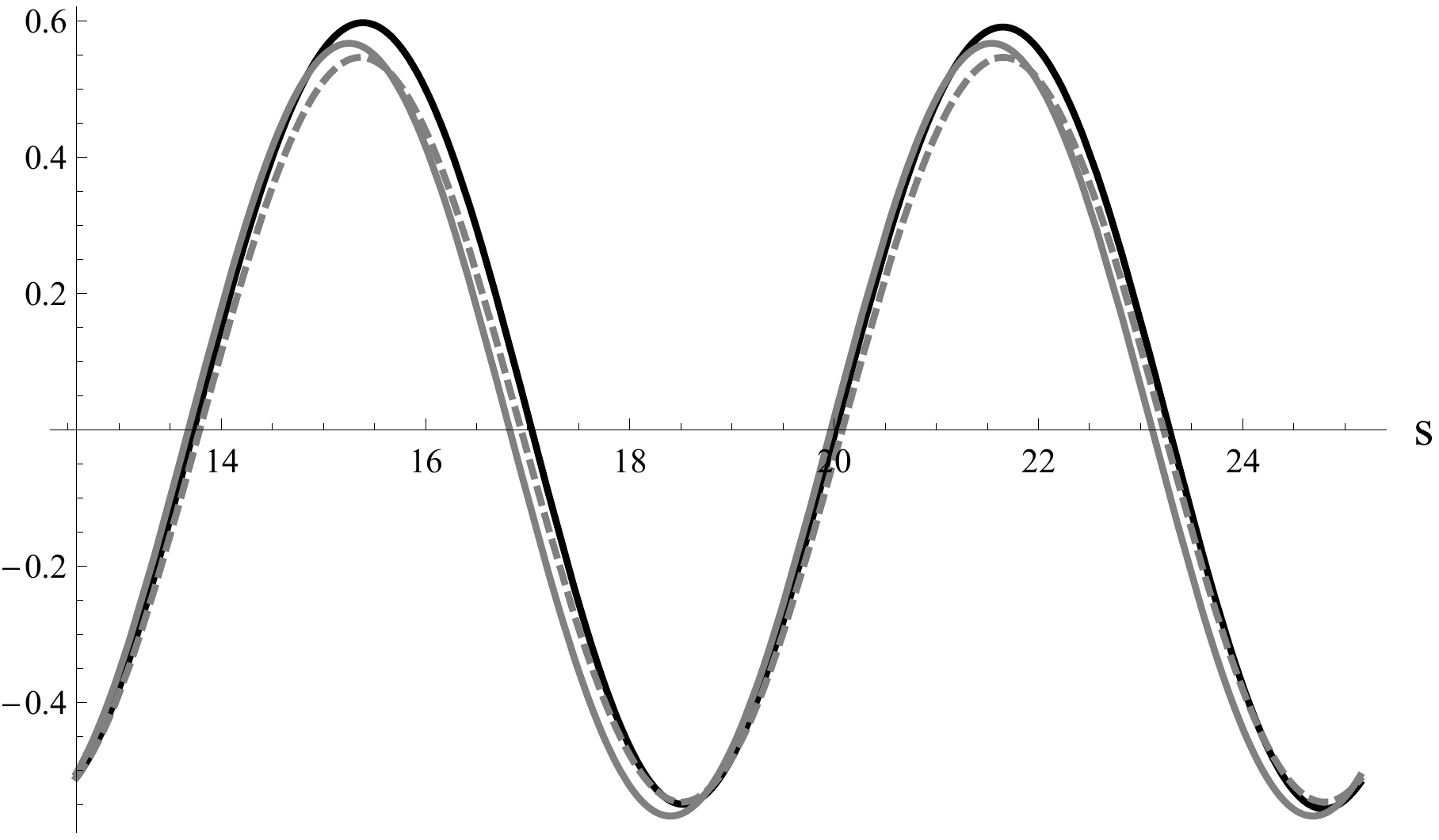}
	\caption{The asymptotic solutions $A_0 \sin{\tau s}+B_0\cos{\tau s}$ (dashed gray line) and  $A_6 \sin{\tau s}+B_6\cos{\tau s}$ (gray line) with  $T=6.291$ (left) and $T=12.582$ (right). The solid black line represents the numerical solution of \eqref{eq:IntegralForm} with $\epsilon=0.1$ .}
	\label{fig:AsymBehaviorSmallE}
\end{figure}
The asymptotic solution \eqref{AsymBehavior}  of eq. \eqref{eq:IntegralForm} describes the exact solution very well, especially for large value of T. As a illustration of the above result we present some numerical simulations. The asymptotic solution \eqref{AsymBehavior} with $n=0$ and $n=6$ is depicted in  Figs. \ref{fig:AsymBehaviorLargeE} and \ref{fig:AsymBehaviorSmallE}. For simplicity we used a finite sum in formulas \eqref{AsymBehaviorCoeff} to calculate an adequate coefficients. Hence, in numerical analysis it is more convenient to use the following approximate expression for $A_n$ and $B_n$
\begin{equation}
    \begin{split}
        A_{n+1}\approx& \frac{v'(T
    ) \cos (T \tau )}{\tau }+v(T) \sin (T \tau )-\frac{1}{\tau}\sum\limits_{m=1}^{N}\, \sum\limits_{\substack{j,k\\ j+k=m+1}}\binom{m+1}{j}\sigma_m A_{n}^{j}B_{n}^{m+1-j}J_1,\\
        B_{n+1}\approx& v(T) \cos (T \tau )-\frac{v'(T) \sin (T \tau )}{\tau }+\frac{1}{\tau}\sum\limits_{m=1}^{N}\sum\limits_{\substack{j,k\\ j+k=m+1}}\binom{m+1}{j}\sigma_m A_{n}^{j}B_{n}^{m+1-j}J_2,
    \end{split}
\end{equation}
where $J_1$ and $J_2$ are defined as before and $N=30$ in our analysis.

\section{Conclusion}
We have provided the rigorous asymptotic analysis of a governing equation modelling the capillary rise phenomenon in a vertical narrow tube. The perturbation analysis of the main equation was a nontrival task due the a appearance of a non-Lipschitz component in the equation. However, the new approach was proposed where the convergence of the exact solution to the unperturbed solution was considered separately on two intervals dependent on the decreasing parameter. For one interval the nonlinear component of the governing equation is not a Lipschitz continuous function. To deal with this problem we use the fact that the length of that interval decreases with $\epsilon$ what allow us to estimate the remainder from above by a function decreasing to zero as $\epsilon$ goes to zero. We conclude the final statement performing a comprehensive calculations. 

In the subsequent section we investigated the asymptotic behavior of a solution for large values of independent variable. The method of finding successive approximation to the exact oscillatory solution was introduced. Finally, the numerical analysis was presented to compare the numerical solution of the governing equation with the approximate periodic asymptotic solutions. 

\bibliography{oscillations}
\bibliographystyle{plain}
\end{document}